\documentclass{amsart}[12pt]

\usepackage{amsmath}
\usepackage{amsfonts}
\usepackage{amssymb}
\usepackage{cite}
\usepackage[usenames]{color}
\usepackage{amsxtra}
\usepackage{amscd}

\newcommand{\floor}[1]{\lfloor #1 \rfloor}

\newcommand{\ID}{{\mathbb D}}
\newcommand{\IR}{{\mathbb R}}

\newtheorem{theorem}{Theorem}
\newtheorem{lemma}{Lemma}

\theoremstyle{definition}

\begin{document}
\title[Complex symmetric  weighted Composition Differentiation Operators]%
{Complex symmetric  weighted Composition Differentiation Operators}

\author{Junming Liu, Saminathan Ponnusamy,  Huayou Xie*}
\thanks{*Corresponding author}

\address{School of Applied Mathematics, Guangdong University of Technology, Guangzhou, Guangdong, 510520, P.~R.~China}\email{jmliu@gdut.edu.cn}

\address{S. Ponnusamy, Department of Mathematics, Indian Institute of Technology Madras, Chennai-600 036, India.} \email{samy@iitm.ac.in}

\address{School of Applied Mathematics, Guangdong University of Technology, Guangzhou, Guangdong, 510520, P.~R.~China}\email{1425100934@qq.com}

\begin{abstract}
In this note,  we completely characterize complex symmetric weighted composition differentiation operator  on the Hardy space $H^2$ with
respect to the conjugation operator $C_{\lambda,\alpha}$. Meanwhile, the normal and self-adjoint of the weighted composition differentiation
operators on the Hardy space $H^2$ are also studied.  This note could be considered as a continuation of the work initiated by Fatehi and Hammond.
\end{abstract}

\thanks{This work was supported by NNSF of China (Grant No. 11801094).}
\keywords{weighted composition differentiation operator, complex symmetric operator, Hardy space}
\subjclass[2010]{47B38, 30H10,47A05}

\maketitle

%\pagestyle{myheadings}
%\markboth{Junming Liu,  Saminathan Ponnusamy and Huayou Xie}{Complex symmetric  weighted Composition Differentiation Operators}

\section{Introduction and preparation}
In this paper, $\mathbb{D}$ denotes the open  unit disc $\{z\in \mathbb{C}:\,|z|<1\}$ and $\mathbb{T}$  the unit circle $\{z\in \mathbb{C}:\,|z|=1\}$.
Let $H(\mathbb{D})$ be the Hilbert space of all analytic function on $\mathbb{D}$. The space $H^2$, the Hardy space, is the  set of functions from
$H(\mathbb{D})$ with square summable power series coefficients; that is,  $f\in H(\mathbb{D})$ for which
$$\|f\|_{H^{2}}=\Big(\sum^{\infty}_{n=0}|a_n|^2\Big)^{1/2}<\infty,
$$
where $\{a_{n}\}$ is the sequence of Maclaurin coefficients for $f$.

Given formal power series  $f(z)=\sum^{\infty}_{n=0}a_nz^n$ and $g(z)=\sum^{\infty}_{n=0}b_nz^n$, the inner product on $H^2$ is defined by
$$\langle f,g \rangle =\sum^{\infty}_{n=0}a_n\overline{b_n},
$$
Let $H^{\infty}$ denote the space of bounded analytic functions on $\mathbb{D}$. The norm $\|\,.\, \|_\infty$ of $f\in H^{\infty}$  is defined by
$$\|f\|_{\infty}=\sup\{|f(z)|:z\in \mathbb{D}\}.
$$

For an analytic self-map $\varphi$ of $\mathbb{D}$, we define the composition operator $C_{\varphi}$ on $H(\mathbb{D})$ by
$$(C_{\varphi}f)(z)= (f\circ \varphi)(z) =f(\varphi(z)), \quad z\in \mathbb{D}.
$$
%for $f\in H(\mathbb{D})$.
This is  the first setting in which composition operators were studied. By Littlewood's subordination principle, every composition operator
takes $H^2$ into itself. The differentiation operator is defined by $Df=f'$ for each $f\in H(\mathbb{D})$. For $u\in H(\mathbb{D})$, the weighted
composition operator $uC_{\varphi}$ is given by
$$(uC_{\varphi}f)(z)=u(z)f(\varphi(z)), \quad f\in H(\mathbb{D}).
$$

For  $m\in \mathbb{N}$, the weighted differentiation composition operator is denoted by
$$(D^m_{u,\varphi}f)(z)=u(z)f^{(m)}(\varphi(z)), \quad f\in H(\mathbb{D}).
$$
When $m=0$, the operator $D^{m}_{u,\varphi}$ becomes the weighted composition operator $uC_{\varphi}$. If $m=0$ and $u(z)=1$,
we get $D^{m}_{u,\varphi}=C_{\varphi}$. If $m=1$ and $u(z)=1$,  it turns out to be $D^{m}_{u,\varphi}=C_{\varphi}D$. If $m=1$ and $u(z)=\varphi'(z)$,
it yields that $D^{m}_{u,\varphi}=DC_{\varphi}$. In this article, we study the case  $m=1$.  In what follows, we denote
$D_{u,\varphi}^{(1)}$ by $D_{u,\varphi}$ for convenience.

Weighted composition operators have arisen in the study of isometries of Hardy spaces. Later, these have been studied by  many mathematicians.
Recently, many researchers have started investigating weighted composition differentiation operator on various function spaces.
For example,   Ohno \cite{SO} studied the boundedness and compactness of the products of composition and differentiation between Hardy spaces whereas
Li and Stevi\'{c} \cite{LS1,LS2} investigated  the products of composition and differentiation operators between $H^{\infty}$ and Bloch type spaces.
Liang and Zhou studied them on logarithmic Bloch space in \cite{LZ}.  Fatehi and Hammond \cite{F} investigated the adjoint, norm and spectrum of the composition
differentiation operator $D_{\varphi}$  on Hardy spaces. They only considered the case of  $\|\varphi\|_{\infty}<1$. In this case, it could guarantee
that $D_{\varphi}$ is bounded and compact. To ensure that $D_{u,\varphi}$ is bounded and compact on Hardy space, we study the operator $D_{u,\varphi}$  with
$u\in H^{2}$ and $\|\varphi\|_{\infty}<1$.

The reproducing kernel of $H^2$ is
$$K_{w}(z)=\frac{1}{1-\overline{w}z}.
$$
For $z\in \mathbb{D}$, the normalized reproducing kernels of $H^2$ are given by
$$k_{w}(z)=\frac{K_{w}(z)}{\|K_{w}(z)\|}=\frac{(1-|w|^2)^{\frac{1}{2}}}{1-\overline{w}z}, \quad z\in \mathbb{D}.
$$
Then $\langle f(z),K_{w}(z) \rangle =f(w)$ for all $f\in H^2$ and $z\in \mathbb{D}$. Now, we introduce
$$K^{(1)}_{w}(z)=\frac{z}{(1-\overline{w}z)^2}.
$$
Then $K^{(1)}_{w}$ is the reproducing kernel for point-evaluation of the first derivative. In fact,
for all $f\in H^2$, we have
$$\langle f(z),K^{(1)}_{w}(z) \rangle =f'(w)
$$
and it is easy to see that
$$\langle f,D^{\ast}_{u,\varphi}(K_w) \rangle =\langle D_{u,\varphi}(f),K_w \rangle =\overline{u(w)}f'(\varphi(w))
=\langle f,\overline{u(w)}K^{(1)}_{\varphi(w)} \rangle .
$$
Therefore, $D^{\ast}_{u,\varphi}(K_w)=\overline{u(w)}K^{(1)}_{\varphi(w)}$.

A linear operator $C:~H^2\rightarrow H^2$ is a {\it conjugation} if
$$\langle Cf,Cg \rangle =\langle g,f \rangle ,\ \ \mbox{for all }\ f, g\in H^2
$$
and $C^2=I$, where $I$ is identity operator. A bounded operator $T$ on the Hardy space $H^2$ is said to be {\it complex symmetric} if
there is a conjugation $C$ on $H^2$ such that $T=CT^{\ast}C$.

Complex symmetric operator  can be regarded  as a generalization of complex symmetric matrices. It is important to the development of
operator theory. The study of complex symmetric operators was started by  Garcia and  Putinar \cite{G,GS}, Garcia and Poore \cite{GP},
and Garcia and  Wogen \cite{GW,GWR}.
Recently, the study of complex symmetric weighted composition operator on different spaces  has attracted  the interest of many researchers.
In $2018$, Lim and Khoi \cite{R} studied the weighted composition operator on the Hilbert space $\mathcal{H}_{\gamma}(\mathbb{D})$ of holomorphic
functions, which is complex symmetric operator with the conjugation of the form $\mathcal{A}_{u,v}f=u\cdot\overline{f\circ \overline{v}}$.
And they also obtained the result about the conjugation $\mathcal{A}_{u,v}$.  Hu et al.  \cite{HYZ} investigated complex symmetric
weighted composition operator on Dirichlet spaces and Hardy spaces.
Wang and Yao \cite{WY}, and Wang and Han \cite{W}  studied complex symmetry of weighted composition operators in several variables.
Hai and Khoi \cite{HP,HV} characterized  complex symmetry of weighted composition operators on the Fock space.
 For further details, we suggest the readers to refer the articles \cite{GZ,HP1,R,NST,Waleed-2017}.

In this paper, we consider the problem of describing all complex symmetric weighted composition differentiation operator on the
Hardy space $H^2$ with the conjugation $C_{\lambda,\alpha}$. The normal and self adjoint properties of weighted composition differentiation
operator is also discussed.

\section{Complex symmetric operator}

For $\lambda,\alpha\in \mathbb{T}$, we define the conjugate linear operator $C_{\lambda,\alpha}$ on the Hardy space $H^2$ by
$$C_{\lambda,\alpha}f(z)=\lambda \overline{f( \overline{\alpha z})},
$$
where $f\in H^2$. It is easy to see that $C_{\lambda,\alpha}$ is a conjugation. For $\alpha=1$, we denote $C_{\lambda}$ by
$$C_{\lambda}f(z)=\lambda \overline{f(\overline{z})}.
$$
In this section, we will characterize complex symmetric  weighted composition differentiation operator on the Hardy space $H^2$.
%We can obtain the condition when $D_{u,\varphi}$ is complex symmetric and give the relation between complex symmetric and self-adjoint.

\begin{theorem}
Suppose that   $\varphi(z)$  is an analytic self-map  on $\mathbb{D}$ such that $\|\varphi\|_{\infty}<1$, and $u\in H^{2}$ with $u\neq 0$. Then $D_{u,\varphi}$ is a complex symmetric operator on $H^2$ with the conjugation $C_{\lambda,\alpha}$ if and only if
there are complex numbers $a,b,c $ such that
$$u(z)=\frac{az}{(1-\alpha bz)^2} ~\mbox{ and }~  \varphi(z)=b+\frac{cz}{1-\alpha bz}, \ \mbox{for all}\  z\in \mathbb{D}.
$$
\end{theorem}
\begin{proof}
Suppose that  $D_{u,\varphi}$ is $C_{\lambda,\alpha}$-symmetric. Then
\begin{equation}\label{LPX-eq1}
D_{u,\varphi}C_{\lambda,\alpha}K_{w}(z)=C_{\lambda,\alpha}D^{*}_{u,\varphi}K_{w}(z)
\end{equation}
for all $w,z\in\mathbb{D}$. Thus,
\begin{align*}
D_{u,\varphi}C_{\lambda,\alpha}K_{w}(z)
=D_{u,\varphi}C_{\lambda,\alpha}\Big (\frac{1}{1-\overline{w}z}\Big )
=D_{u,\varphi}\Big ( \frac{\lambda}{1-\alpha wz} \Big )
=\frac{\lambda\alpha wu(z)}{(1-\alpha w \varphi(z))^2}
\end{align*}
and
\begin{align*}
C_{\lambda,\alpha}D^{*}_{u,\varphi}K_{w}(z)
=C_{\lambda,\alpha}D^{*}_{u,\varphi} \Big (\frac{1}{1-\overline{w}z} \Big )
=C_{\lambda,\alpha}\Big (\frac{z\overline{u(w)}}{(1-z\overline{\varphi(w)})^2}\Big )
=\frac{\lambda\alpha zu(w)}{(1-\alpha z \varphi(w))^2}.
\end{align*}
%\begin{align*}
%D_{u,\varphi}C_{\lambda,\alpha}K_{w}(z)&=D_{u,\varphi}C_{\lambda,\alpha}\Big (\frac{1}{1-\overline{w}z}\Big ) \\
%&=D_{u,\varphi}\Big ( \frac{\lambda}{1-\alpha wz} \Big ) \\
%&=\frac{\lambda\alpha wu(z)}{(1-\alpha w \varphi(z))^2}
%\end{align*}
%and
%\begin{align*}
%C_{\lambda,\alpha}D^{*}_{u,\varphi}K_{w}(z)&=C_{\lambda,\alpha}D^{*}_{u,\varphi} \Big (\frac{1}{1-\overline{w}z} \Big ) \\
%&=C_{\lambda,\alpha}\Big (\frac{z\overline{u(w)}}{(1-z\overline{\varphi(w)})^2}\Big ) \\
%&=\frac{\lambda\alpha zu(w)}{(1-\alpha z \varphi(w))^2}.
%\end{align*}
In view of \eqref{LPX-eq1}, it follows that
\begin{equation}\label{eq1}
\frac{\lambda\alpha wu(z)}{(1-\alpha w \varphi(z))^2}=\frac{\lambda\alpha zu(w)}{(1-\alpha z \varphi(w))^2}
\end{equation}
for all $w,z\in \mathbb{D}$. Since $\lambda,\alpha\in \mathbb{T}$ and $u\neq 0$, we must have $u(0)=0$.

Now, we set $u(z)=\sum^{\infty}\limits_{n=1}a_{n}z^n$ with $a_{n}\in \mathbb{C}$. Substituting $u(z)$ back into the
equation \eqref{eq1},  we get
$$\Big (\sum^{\infty}_{n=1}a_nwz^n\Big )(1-\alpha z\varphi(w))^2 = \Big (\sum^{\infty}_{n=1}a_n zw^n\Big )(1-\alpha w\varphi(z))^2
$$
for all $z,w\in \mathbb{D}$. Differentiating the above formula with respect to $w$, we obtain
\begin{align*}
\Big (\sum^{\infty}_{n=1}& a_nz^n \Big)(1-\alpha z\varphi(w))^2-2\alpha z\varphi'(w)\Big (\sum^{\infty}_{n=1}a_nwz^n \Big)(1-\alpha w\varphi(w))\\
&= \Big (\sum^{\infty}_{n=1}na_nzw^{n-1}\Big )(1-\alpha w\varphi(z))^2-2\alpha \varphi(z) \Big (\sum^{\infty}_{n=1}a_nzw^n \Big)(1-\alpha w\varphi(z)).
\end{align*}
Let $w=0$ in the above equation.  Then, we have
$$u(z)=\frac{az}{(1-\alpha bz)^2}, \ \mbox{for all}~ z\in \mathbb{D},
$$
where $a=a_1$ and $b=\varphi(0).$
Substituting $u(z)=\frac{az}{(1-\alpha bz)^2}$ back into the equation \eqref{eq1}, we have
$$(1-\alpha bz)^2(1-\alpha w\varphi(z))^2=(1-\alpha bw)^2(1-\alpha z\varphi(w))^2
$$
for all  $z,w\in \mathbb{D}$.  Differentiate both sides of the above equation with respect to $w$, we see that
\begin{align*}
&(1-\alpha bz)^2(1-\alpha w\varphi(z))(-2\alpha \varphi(z)) \\
&=(1-\alpha bw)(1-\alpha z\varphi(w))^2(-2\alpha b) +(1-\alpha z\varphi(w))(1-\alpha bw)^2(-2\alpha z\varphi'(w)).
\end{align*}
Letting $w=0$ in the above formula  shows that
$$\varphi(z)=b+\frac{cz}{1-\alpha bz},\ \mbox{for all}\ z\in \mathbb{D},
$$
where $c=\varphi'(0)$.

Conversely, if $u(z)=\frac{az}{(1-\alpha bz)^2}$ and $\varphi(z)=b+\frac{cz}{1-\alpha bz}$,
then
$$C_{\lambda,\alpha}K_{w}(z)= C_{\lambda,\alpha}\Big (\frac{1}{1-\overline{w}z}\Big )=\frac{\lambda}{1-\alpha wz}
$$
and thus,
\begin{equation}\label{LPX-eq2}
D_{u,\varphi}C_{\lambda,\alpha}K_{w}(z)
=\frac{\lambda \alpha wu(z)}{(1-\alpha w\varphi(z))^2}
=\frac{\lambda\alpha awz}{(1-\alpha bz-\alpha bw+\alpha^2b^2zw-\alpha czw)^2}.
\end{equation}
%\begin{align*}
%D_{u,\varphi}C_{\lambda,\alpha}K_{w}(z)&=D_{u,\varphi}C_{\lambda,\alpha}\Big (\frac{1}{1-\overline{w}z}\Big ) \\
%&=D_{u,\varphi} \Big (\frac{\lambda}{1-\alpha wz} \Big )    \\
%&=\frac{\lambda \alpha wu(z)}{(1-\alpha w\varphi(z))^2}  \\
%&=\frac{\lambda \alpha w \frac{az}{(1-\alpha bz)^2}}{\left(1-\alpha w\big (b+\frac{cz}{1-\alpha bz}\big )\right)^2}    \\
%&=\frac{\lambda\alpha awz}{(1-\alpha bz-\alpha bw+\alpha^2b^2zw-\alpha czw)^2}
%\end{align*}
%and
%\begin{align*}
%C_{\lambda,\alpha}D_{u,\varphi}^{*}K_{w}(z)&=C_{\lambda,\alpha}D_{u,\varphi}^{*}(\frac{1}{1-\overline{w}z}) \\
%&=C_{\lambda,\alpha}(\frac{z\overline{u(w)}}{1-\overline{\varphi(w)}z})    \\
%&=\frac{\lambda \alpha zu(w)}{(1-\alpha z\varphi(w))^2}  \\
%&=\frac{\lambda \alpha z \frac{aw}{(1-\alpha bw)^2}}{\left(1-\alpha z(b+\frac{cw}{1-\alpha bw})\right)^2}    \\
%&=\frac{\lambda\alpha awz}{(1-\alpha bw-\alpha bz+\alpha^2b^2zw-\alpha czw)^2}.
%\end{align*}
Again, as
$$ D_{u,\varphi}^{*}K_{w}(z)=D_{u,\varphi}^{*}\Big (\frac{1}{1-\overline{w}z}\Big)=\frac{z\overline{u(w)}}{1-\overline{\varphi(w)}z} ,
$$
it follows that
\begin{equation}\label{LPX-eq3}
C_{\lambda,\alpha}D_{u,\varphi}^{*}K_{w}(z)
=\frac{\lambda \alpha zu(w)}{(1-\alpha z\varphi(w))^2}
%=\frac{\lambda \alpha z \frac{aw}{(1-\alpha bw)^2}}{\left(1-\alpha z(b+\frac{cw}{1-\alpha bw})\right)^2}
=\frac{\lambda\alpha awz}{(1-\alpha bw-\alpha bz+\alpha^2b^2zw-\alpha czw)^2}.
\end{equation}
 Comparing \eqref{LPX-eq2} and \eqref{LPX-eq3} shows that   $D_{u,\varphi}C_{\lambda,\alpha}=C_{\lambda,\alpha}D_{u,\varphi}^{*}$ and
hence, $D_{u,\varphi}$ is $C_{\lambda,\alpha}$-symmetric.
\end{proof}

The above theorem can be seen as a continuation of the work of \cite[Theorem 3.3]{SJ} and \cite[Proposition 2.9]{GH}.

\section{Normal and self-adjoints}
In the section, we study the  adjoint of $D_{u,\varphi}$.
For  $z\in \mathbb{D}$, let
\begin{equation}\label{LPX-eq4}
\varphi(z)=\frac{az+b}{cz+d} ~\mbox{ and }~
\sigma(z)=\frac{\overline{a}z-\overline{c}}{-\overline{b}z+\overline{d}}.
\end{equation}
If $\varphi$ is an analytic self-map of $\ID$, then so does $\sigma$, see \cite{C}.
The following proposition gives the adjoint of the weighted differentiation composition operator.

\begin{lemma}
If  $\varphi$ and $\sigma$ are linear fractional self-maps of $\ID$ as in \eqref{LPX-eq4},
%$$\varphi(z)=\frac{az+b}{cz+d} ~\mbox{ and }~ \sigma(z)=\frac{\overline{a}z-\overline{c}}{-\overline{b}z+\overline{d}}
%$$
where ${\|\varphi\|}_{\infty}<1$, then $D_{K^{(1)}_{\sigma(0)},\varphi}^{\ast}=D_{K^{(1)}_{\varphi(0)},\sigma}.$
\end{lemma}
\begin{proof}
For $z\in \mathbb{D}$, it is easy to obtain that
$$K^{(1)}_{\varphi(0)}(z)=\frac{\overline{d}^2z}{(\overline{d}-\overline{b}z)^2} ~\mbox{ and }~
K^{(1)}_{\sigma(0)}(z)=\frac{d^2z}{(d+cz)^2}.
$$
By simple calculations, we obtain that, for $z,w\in\mathbb{D}$,
\begin{align*}
D_{K^{(1)}_{\varphi(0)},\sigma}K_{w}(z)
=D_{K^{(1)}_{\varphi(0)},\sigma} \Big( \frac{1}{1-\overline{w}z} \Big)
=\frac{\overline{d}^2z}{(\overline{d}-\overline{b}z)^2}\cdot\frac{\overline{w}}{\left(1-\overline{w}\Big (\frac{\overline{a}z-\overline{c}}{-\overline{b}z+\overline{d}}\Big)\right)^2}
%=&\frac{\overline{d}^2\overline{w}z}{\left(\overline{d}+\overline{cw}-\overline{aw}z-\overline{b}z\right)^2} \\
\end{align*}
and
\begin{align*}
D_{K^{(1)}_{\sigma(0)},\varphi}^{\ast}K_{w}(z)
=D_{K^{(1)}_{\sigma(0)},\varphi}^{\ast}\Big (\frac{1}{1-\overline{w}z} \Big)
=\frac{\overline{d}^2\overline{w}}{(\overline{d}+\overline{cw})^2}\cdot\frac{z}{\left(1-z\left(\frac{\overline{aw}+\overline{b}}{\overline{cw}+\overline{d}}\right)\right)^2}
%& =\frac{\overline{d}^2\overline{w}z}{\left(\overline{d}-\overline{b}z-\overline{w}\overline{a}z+\overline{wc}\right)^2}
\end{align*}
which upon   simplifications of the right hand side of the last two relations give the desired result.
%shows that $D_{K^{(1)}_{\sigma(0)},\varphi}^{\ast}=D_{K^{(1)}_{\varphi(0)},\sigma}.$
%\begin{align*}
%D_{K^{(1)}_{\varphi(0)},\sigma}K_{w}(z)=&D_{K^{(1)}_{\varphi(0)},\sigma} \Big( \frac{1}{1-\overline{w}z} \Big)  \\
%=&\frac{\overline{d}^2z}{(\overline{d}-\overline{b}z)^2}\cdot\frac{\overline{w}}{\left(1-\overline{w}\Big (\frac{\overline{a}z-\overline{c}}{-\overline{b}z+\overline{d}}\Big)\right)^2}  \\
%=&\frac{\overline{d}^2\overline{w}z}{\left(\overline{d}+\overline{cw}-\overline{aw}z-\overline{b}z\right)^2} \\
%\end{align*}
%and
%\begin{align*}
%D_{K^{(1)}_{\sigma(0)},\varphi}^{\ast}K_{w}(z)& =D_{K^{(1)}_{\sigma(0)},\varphi}^{\ast}(\frac{1}{1-\overline{w}z}) \\
%& =\frac{\overline{d}^2\overline{w}}{(\overline{d}+\overline{cw})^2}\cdot\frac{z}{\left(1-z\left(\frac{\overline{aw}+\overline{b}}{\overline{cw}+\overline{d}}\right)\right)^2} \\
%& =\frac{\overline{d}^2\overline{w}z}{\left(\overline{d}-\overline{b}z-\overline{w}\overline{a}z+\overline{wc}\right)^2}
%\end{align*}
%where $z,w\in\mathbb{D}$.
The proof is complete.
\end{proof}

In the above lemma, we used the weighted composition differentiation  operator to express the adjoint formula given by Fatehi and Hammond \cite{F}.
A bounded linear operator $T$ in  Hardy spaces  is called a normal operator if $T^{\ast}T=TT^{\ast}$.  In the following theorem, we discuss when the
weighted composition differentiation operator is normal.

\begin{theorem}
For  $a,b,c,d\in \mathbb{C}$ such that $ad-bc\neq0$, suppose that $\varphi$ and $\sigma$ are given by \eqref{LPX-eq4},
%$$\varphi(z)=\frac{az+b}{cz+d},\sigma(z)=\frac{\overline{a}z-\overline{c}}{-\overline{b}z+\overline{d}}
%$$
where ${\|\varphi\|}_{\infty}<1$. Let $u$ be defined as $u=K_{\sigma(0)}^{(1)}$. If $a\overline{b}=-\overline{a}c$ and  $\varphi(0)=\sigma(0)$, then $D_{u,\varphi}$ is normal.
\end{theorem}
\begin{proof}
To verify that  $D_{u,\varphi}$ is normal, we only need to prove that
$$D_{u,\varphi}^{\ast}D_{u,\varphi}K_{w}(z)=D_{u,\varphi}D_{u,\varphi}^{\ast}K_{w}(z) ~\mbox{ for all $z\in \mathbb{D}$.}
$$
For the convenience, we let
$$\left\{
\begin{aligned}
|a|^2-|b|^2=t_1 \\
b\overline{d}-a\overline{c}=t_2 \\
\overline{a}c-\overline{b}d=t_3 \\
|d|^2-|c|^2=t_4
\end{aligned}
\right.
\ \ \text{ and }\ \
\left\{
\begin{aligned}
|a|^2-|c|^2=k_1 \\
\overline{a}b-\overline{c}d=k_2 \\
c\overline{d}-a\overline{b}=k_3\\
|d|^2-|b|^2=k_4.
\end{aligned}
\right.
$$

If $\varphi(0)=\sigma(0)$, then we can easily obtain that $\overline{b}d=-c\overline{d}$ and $|c|^2=|b|^2$. Since $a\overline{b}=-\overline{a}c$,
it follows easily that $t_i=k_i$   for   $i=1,2,3,4$. By calculation, we find that
\begin{align*}
%&(\overline{a}z-\overline{c})(\overline{d}-\overline{b}z)(ad-bc)  \\
(\overline{a}z-\overline{c})(\overline{d}-\overline{b}z)(ad-bc)
=&\left [-\overline{ab}z^2+(\overline{ad}+\overline{bc})z-\overline{cd}\right ](ad-bc)  \\
=&\left (|b|^2\overline{a}c-|a|^2\overline{b}d\right )z^2+\left (|ad|^2-\overline{ad}bc+\overline{bc}ad-|bc|^2 \right )z+ |c|^2b\overline{d}-|d|^2a\overline{c}  \\
=&\left (|a|^2c\overline{d}-|c|^2a\overline{b}\right )z^2+\left (|ad|^2-ad\overline{bc}+\overline{ad}bc-|bc|^2\right )z+|d|^2\overline{a}b-|b|^2\overline{c}d  \\
=&\left (\overline{ad}-\overline{bc}\right )acz^2+\left (\overline{ad}-\overline{bc}\right )(ad+bc)z+ \left (\overline{ad}-\overline{bc}\right )bd  \\
%=& \left (\overline{ad}-\overline{bc}\right )\left [acz^2+(ad+bc)z+bd \right ]  \\
=&(az+b)(cz+d)(\overline{ad}-\overline{bc} )
\end{align*}
for all $z\in \mathbb{D}$. Since
$$D_{K^{(1)}_{\sigma(0)},\varphi}^{\ast}=D_{K^{(1)}_{\varphi(0)},\sigma},
$$
we have
\begin{align*}
D_{u,\varphi}^{\ast}D_{u,\varphi}K_{w}(z)&=D_{K^{(1)}_{\varphi(0)},\sigma}D_{K^{(1)}_{\sigma(0)},\varphi}\left(\frac{1}{1-\overline{w}z}\right)  \\
&=D_{K^{(1)}_{\varphi(0)},\sigma}\left(\frac{d^2z\overline{w}}{(d+cz)^2(1-\overline{w}\varphi(z))^2}\right)   \\ &=|d|^4z\overline{w}\cdot\frac{(d-c\sigma(z))}{(\overline{d}-\overline{b}z)^2(d+c\sigma(z))^3(1-\overline{w}\varphi(\sigma(z)))^2}      \\
&\hspace{1cm} +|d|^4z\overline{w}\cdot\frac{2\overline{w}\sigma(z)\varphi'(\sigma(z))}{(\overline{d}-\overline{b}z)^2(d+c\sigma(z))^2(1-\overline{w}\varphi(\sigma(z)))^3}  \\
&=|d|^4z\overline{w}\bigg(\frac{|d|^2+|c|^2-(\overline{a}c+\overline{b}d)z}{(t_3z+t_4)[(t_3z+t_4)-(t_1z+t_2)\overline{w}]^2}      \\
&\hspace{1cm} +\frac{2(\overline{a}z-\overline{c})(\overline{d}-\overline{b}z)(ad-bc)\overline{w}}{(t_3z+t_4)[(t_3z+t_4)-(t_1z+t_2)\overline{w}]^3}\bigg)  \\
&=|d|^4z\overline{w}\bigg(\frac{|d|^2+|b|^2+(a\overline{b}+c\overline{d})z}{(k_3z+k_4)[(k_3z+k_4)-(k_1z+k_2)\overline{w}]^2}      \\
& \hspace{1cm}+\frac{2(az+b)(cz+d)(\overline{ad}-\overline{bc})\overline{w}}{(k_3z+k_4)[(k_3z+k_4)-(k_1z+k_2)\overline{w}]^3}\bigg)  \\
&=|d|^4z\overline{w}\cdot\frac{(\overline{d}+\overline{b}\varphi(z))}{(d+cz)^2(\overline{d}-\overline{b}\varphi(z))^3(1-\overline{w}\sigma(\varphi(z)))^2}           \\
&\hspace{1cm}+|d|^4z\overline{w}\cdot\frac{2\overline{w}\varphi(z)\sigma'(\varphi(z))}{(d+cz)^2(\overline{d}-\overline{b}\varphi(z))^2(1-\overline{w}\sigma(\varphi(z)))^3} \\
&=D_{u,\varphi}D_{u,\varphi}^{\ast}K_{w}(z)
\end{align*}
for all $z,w\in \mathbb{D}$. Hence, $D_{u,\varphi}^{\ast}D_{u,\varphi}=D_{u,\varphi}D_{u,\varphi}^{\ast}$ which in turn implies that $D_{u,\varphi}$ is normal.
\end{proof}

In the above theorem, if we choose $a=i, b=1+i, c=1-i$ and $d=8i$, then the operator $D_{u,\varphi}$ is a normal operator.
However,  $D_{u,\varphi}$ is not a self-adjoint operator.

A bounded linear operator $T$ in a Hardy space is called a self-adjoint  if $T^{\ast}=T$.  In the following theorem, we study when
the weighted composition differentiation operator is self-adjoint.

\begin{theorem}
Let $u(z)$  be a nonzero analytic function in  $H^2$ and $\varphi(z)$ be analytic self-map  on $\mathbb{D}$ such that
$\|\varphi\|_{\infty}<1$. Then $D_{u,\varphi}$ is self-adjoint if and only if
$$u(z)=\frac{az}{(1-bz)^2}~\mbox{ and }~\varphi(z)=b+\frac{cz}{1-bz} ~\mbox{ for all $z\in \mathbb{D}$},
$$
where $a,b,c\in \mathbb{R}$.
\end{theorem}
\begin{proof}
If $u(z)=\frac{az}{(1-bz)^2}$ and $\varphi(z)=b+\frac{cz}{1-bz}$, where $a,b,c\in\mathbb{R}$, for all $z\in\mathbb{D}$, then we have
\begin{align*}
D^{\ast}_{u,\varphi}K_{w}(z)&=D^{\ast}_{u,\varphi}\Big(\frac{1}{1-\overline{w}z} \Big)
%\\&
=\frac{z\overline{u(w)}}{(1-z\overline{\varphi(w)})^2}
%\\&
=\frac{\overline{a}z\overline{w}}{(1-\overline{b}\overline{w}-\overline{b}z+z\overline{w}\overline{b}^2-\overline{c}z\overline{w})^2}
\end{align*}
and
\begin{align*}
D_{u,\varphi}K_{w}(z)&=D_{u,\varphi}(\frac{1}{1-\overline{w}z})
%\\&
=\frac{\overline{w}u(z)}{(1-\overline{w}\varphi(z))^2}
%\\&
=\frac{az\overline{w}}{(1-bz-b\overline{w}+z\overline{w}b^2-cz\overline{w})^{2}},
\end{align*}
where $z,w\in \mathbb{D}$. Hence, $D_{u,\varphi}^{\ast}=D_{u,\varphi}$ which shows that $D_{u,\varphi}$ is self-adjoint.

 Conversely, we suppose that $D_{u,\varphi}$ is self-adjoint. Then
$D_{u,\varphi}^{*}=D_{u,\varphi}$.
This implies that $D_{u,\varphi}^{\ast}K_{w}(z)=D_{u,\varphi}K_{w}(z)$ for $w,z\in\mathbb{D}$.
Notice that
$$D_{u,\varphi}^{*}K_{w}(z)=\frac{\overline{u(w)}z}{(1-z\overline{\varphi(w)})^2} ~\mbox{ and }~
D_{u,\varphi}K_{w}(z)=\frac{u(z)\overline{w}}{\left(1-\overline{w}\varphi(z)\right)^2}.
$$
Hence,
\begin{equation}\label{eq2}
\frac{\overline{u(w)}z}{(1-z\overline{\varphi(w)})^2}=\frac{u(z)\overline{w}}{\left(1-\overline{w}\varphi(z)\right)^2}
\end{equation}
for $z,w\in \mathbb{D}$.
With $w=0$, we show that $u(0)=0$.

For $u\in H^2$, we let  $u(z)=\Sigma^{\infty}_{n=1}a_{n}z^n$ with $a_n\in \mathbb{C}$. Substitute $u(z)$ back into \eqref{eq2} to obtain
\begin{equation}\label{eq3}
\Big (\sum^{\infty}_{n=1}\overline{a_n}z\overline{w}^n\Big )(1-\overline{w}\varphi(z))^2= \Big (\sum^{\infty}_{n=1}a_n\overline{w}z^n\Big )(1-z\overline{\varphi(w)})^2
\end{equation}
for all $z,w\in \mathbb{D}$.

Now, differentiating the equation \eqref{eq3} with respect to $\overline{w}$, we obtain
\begin{align*}
\Big (\sum^{\infty}_{n=1}&n\overline{a_n}z\overline{w}^{n-1}\Big )(1-\overline{w}\varphi(z))^2
-2\varphi(z)\Big(\sum^{\infty}_{n=1}\overline{a_n}z\overline{w}^n\Big)(1-\overline{w}\varphi(z))\\
&= \Big(\sum^{\infty}_{n=1}a_nz^n\Big)(1-z\overline{\varphi(w)})^2-2z\overline{\varphi'(w)}\Big(\sum^{\infty}_{n=1}a_n\overline{w}z^n\Big)(1-z\overline{\varphi(w)}).
\end{align*}
Set $w=0$. Then
$\overline{a_{1}}z=(1-\overline{\varphi(0)}z)^2u(z)
$
so that
\begin{equation}\label{eq4}
 u(z)=\frac{\overline{a_1}z}{(1-\overline{\varphi(0)}z)^2}=\frac{\overline{a}z}{(1-\overline{b}z)^2},
\end{equation}
with $a=a_1$ and $b=\varphi(0)$.

Substituting \eqref{eq4} back into \eqref{eq2} yields that
\begin{equation}\label{eq5}
a(1-\overline{b}z)^2(1-\overline{w}\varphi(z))^2=\overline{a}(1-b\overline{w})^2(1-z\overline{\varphi(w)})^2,
\end{equation}
where $z\in\mathbb{D}$. Setting $w=0$ shows that $\overline{a}=a$,  i.e., $a\in\mathbb{R}$.

Differentiating the equation \eqref{eq5} with respect to $\overline{w}$ gives
\begin{align*}
&(1-\overline{b}z)^2(1-\overline{w}\varphi(z))(-2a\varphi(z))  \\
&\ \ =(1-z\overline{\varphi(w)})^2(-2ab)(1-b\overline{w})+(1-b\overline{w})^2(1-z\overline{\varphi(w)})(-2az\overline{\varphi'(w)})
\end{align*}
for all $z\in \mathbb{D}$.
If $w=0$, then we have
$$\varphi(z)=b+\frac{cz}{1-bz},~ z\in \mathbb{D},
$$
where $b= \varphi(0)$   and $c=\overline{\varphi'(0)}$.

Now the fact that $D_{u,\varphi}$ is self-adjoint implies that
$$D^{\ast}_{u,\varphi}K_{w}(z)=D_{u,\varphi}K_{w}(z).
$$
Similar to the sufficiency part of the proof of the theorem, we get $b, c\in\mathbb{R}$.
\end{proof}

Recall that a bounded linear operator $T$ on $H^2$ is unitary if and only if
$$TT^{\ast}=T^{\ast}T=I.
$$
In the above theorem, we considered conditions such that the operator $D_{u,\varphi}$  is self-adjoint.
Clearly, $D_{u,\varphi}$ is also normal.  But it is not unitary. We have $D_{u,\varphi}^{\ast}D_{u,\varphi}\neq I$ by {\color{red} a} simple calculation.
Let $u=az$ and $\varphi=cz$,   where  $a,c\in \mathbb{R}$. If $|a|<\infty$ and $0<c<1$, then $D_{u,\varphi}$ is self-adjoint.
In this case, if we can find the spectrum of the operator $D_{u,\varphi}$,  we will obtain the norm of $D_{u,\varphi}$.

\begin{theorem}
Assume that $u(z)=az$ and $\varphi(z)=cz$, where  $a\in \IR$ and $c\in (0,1)$. Then
$$\|D_{u,\varphi}\|=akc^k,
$$
where  $k=\floor{\frac{1}{1-c}}$, and  $\floor{\, \cdot \, }$  denotes the greatest integer function.
\end{theorem}

\begin{proof}
For $u(z)=az$ and $\varphi(z)=cz$, we have
$$D_{u,\varphi}(z^{n-1})=acz(n-1)(cz)^{n-2}=(n-1)ac^{n-1}z^{n-1}
$$
for all $n\in\mathbb{N}=\{1,2, \ldots \}$.  We see that  $\{(n-1)ac^{n-1}:\, n\in \mathbb{N}\} $ belongs to the spectrum of $D_{u,\varphi}$.
Next we let $\lambda$ be an arbitrary eigenvalue of $D_{u,\varphi}$ with   the corresponding eigenvector $f$.
Then
\begin{equation}\label{eq6}
\lambda f(z)=aczf'(cz).
\end{equation}
If $f(0)\neq0$, then $\lambda=0$. If $f(0)=0$,  then  differentiate the  equation \eqref{eq6} with respect to $z$ to obtain
\begin{equation}\label{eq7}
\lambda f'(z)=acf'(cz)+ac^2zf''(cz).
\end{equation}
If $f'(0)\neq0$, then $\lambda=ac$. If $f''(0)=0$, then we differentiate both sides of the equation \eqref{eq7} to get
\begin{equation}\label{eq8}
\lambda f''(z)=2ac^2f''(cz)+ac^3zf'''(cz).
\end{equation}
If $f''(0)\neq 0$, then $\lambda=2ac^2$. Also, $(n-1)$-times differentiation of  the equation \eqref{eq6} gives
$$\lambda f^{(n-1)}(z)=(n-1)ac^{n-1}f^{(n-1)}(cz)+ac^nzf^{(n)}(cz).
$$
If $f^{(n-1)}(0)\neq0$, then $\lambda=(n-1)ac^{(n-1)}$. Therefore, any eigenvalue can be represented in this form $(n-1)ac^{(n-1)}$ with $n\in \mathbb{N}$. Since $D_{u,\varphi}$ is compact,  the spectrum of $D_{u,\varphi}$ is precisely  $\{0\}\bigcup \{(n-1)ac^{(n-1)}: n\in \mathbb{N}\}.$
Hence
$$\|D_{u,\varphi}\|=\max\{(n-1)ac^{(n-1)}: n\in \mathbb{N}\}.
$$
Let $g(x)=xc^{x}$. It  can be found that $g(x)$ has maximum in $[0,\infty)$. To obtain the maximum, we need to find that the greatest natural number $n$ such that
$$(n-2)c^{(n-2)}\leq(n-1)c^{n-1}, ~\mbox{ i.e., }~ n\leq\frac{1}{1-c}+1.
$$
This gives $\max\{(n-1)ac^{(n-1)}\}=akc^k $,  where  $k=\floor{\frac{1}{1-c}}.$
\end{proof}

\vspace{0.3cm}
{\bf Data Availability Statement: } My manuscript has no associated data.

\end{document}